\documentclass[11pt]{amsart}
\usepackage{amsmath,amssymb,amscd,verbatim,color}
\usepackage[english]{babel}
\usepackage[utf8x]{inputenc}
\usepackage[T1]{fontenc}
\usepackage{amsmath}
\usepackage{amsthm}
\usepackage{graphicx}
\usepackage[export]{adjustbox}
\usepackage[colorinlistoftodos]{todonotes}
\usepackage[colorlinks=true, allcolors=blue]{hyperref}
\def\NN{{\mathbb N}}

\def\RR{{\mathbb R}}

\newtheorem{Thm}{Theorem}[section]
\newtheorem{Lem}{Lemma}[section]
\newtheorem{Prop}{Proposition}[section]

\newtheorem{Rmk}{Remark}[section]

\numberwithin{equation}{section}

\allowdisplaybreaks

\def\RR{\mathbb{R}}
\def\NN{\mathbb{N}}

\def\cA{{\mathcal A}}

\def\cF{{\mathcal F}}
\def\cM{{\mathcal M}}

\def\cS{{\mathcal S}}

\def\RR{{\mathbb R}}

\def\NN{{\mathbb N}}

\begin{document}
\title[Synchronization in  dtds-RDS]{Synchronization in discrete-time, discrete-state Random Dynamical Systems}

\author[W. Huang]{W. Huang}
\address{Wen Huang: School of Mathematical Sciences, University of Science and
Technology of China, Hefei, Anhui 230026, PRC}
\email{wenh@mail.ustc.edu.cn}

\author[H. Qian]{H. Qian}
\address{Hong Qian: Department of Applied Mathematics, University of Washington, Seattle, WA 98195,
U.S.A} \email{hqian@u.washington.edu}

\author[S. Wang]{S. Wang$^*$}
\address{Shirou Wang: Department of Mathematical \& Statistical Sci,
University of Alberta, Edmonton, Alberta, Canada T6G2G1} \email{shirou@ualberta.ca}

\author[F. X.-F. Ye]{F. X.-F. Ye$^*$}
\address{Felix X.-F. Ye: Department of Applied Mathematics \& Statistics,
	Johns Hopkins University, Baltimore, MD 21218, U.S.A}
\email{xye16@jhu.edu}

\author[Y. Yi]{Y. Yi}
\address{Yingfei Yi: Department of Mathematical
\& Statistical Sci, University of Alberta, Edmonton, Alberta,
Canada T6G2G1,  and School of Mathematics, Jilin University,
Changchun 130012, PRC} \email{yingfei@ualberta.ca}

\thanks{$^*$ Corresponding authors}
\thanks{The first author was  supported by NSFC grants 11371339, 11431012 and
11731003. Other authors were  supported by a PIMS CRG grant. The
third author was also  partially supported by NSFC grants 11771026, 11471344, 
and acknowledges PIMS PTCS and the PIMS at University of Washington supported by NSF DMS-1712701 and NDF DMS-1444084.
The fifth author was also partially supported by NSERC discovery grant 1257749, PIMS CRG grant, a faculty development grant from the Univ. of Alberta, and a Scholarship from Jilin
Univ. The work is partially done when the first, second and
forth authors were visiting UA and the third author was
visiting USTC and UW}



\begin{abstract} We characterize synchronization phenomenon 
in discrete-time,
discrete-state random dynamical systems, with random and
probabilistic Boolean networks as  particular examples. In terms of
multiplicative ergodic properties of the induced linear cocycle, we
show such a random dynamical system with finite state
synchronizes if and only if the Lyapunov exponent $0$ has simple
multiplicity. For the case of countable state space, characterization
of synchronization is provided in term of the spectral subspace corresponding to the Lyapunov exponent $-\infty$. In addition, for both cases of finite and countable state spaces, the mechanism of partial synchronization is described by partitioning the state set into synchronized subsets. Applications to  biological networks are also discussed.
\end{abstract}
\maketitle

\section{Introduction}

Deterministic dynamics with discrete-time steps in a discrete-state space
has a long tradition since the work of von Neumann on
automata in 1950s \cite{Von}.  The subject was significantly developed in 1970s \cite{Wol} parallel to the rise of nonlinear dynamical systems theory.  For complex dynamics that arise in natural and social sciences, statistical physics employs a stochastic representation of dynamical behavior and phenomena.  Stochastic processes and random dynamical systems (RDS) are two distinctly different types of models that generalize, respectively, traditional differential equations and deterministic dynamical systems: The former represents the stochastic movement of an individual system with {\em intrinsic} uncertainties; the latter describes the motions of many individuals under a common deterministic law that is randomly changing with time due to {\em extrinsic} noises \cite{YWQ}. The dynamics of an RDS may exhibit a counterintuitive phenomenon called noise-induced synchronization: The stochastic motions of noninteracting systems with different initial states synchronize under a common noisy law of motion; their individual trajectories converge to one stochastic motion.

This paper concerns the study of  synchronization phenomenon in a 
discrete-time, discrete-state random dynamical system. More
precisely, let ${\mathcal S=\{s_j\}}$ denote the {\em state set}
which can be either finite or countable and furnished with the
discrete topology. For a given measure-preserving map $\theta$ on a
probability space $(\Omega,\cF,\mu)$, where $\mu$ is a probability
measure defined on the $\sigma$-algebra $\cF$ of $\Omega$, a {\em
discrete-time, discrete-state random dynamical system} ({\em
dtds-RDS} for short) is a random cocycle $\cA(n,\cdot): \cS\to\cS$,
$n\in\NN$, over $(\Omega,\cF,\mu,\theta)$, i.e., for each $n$,
$\{\cA(n,\omega): \omega\in\Omega\}$ is a measurable family of
continuous mappings  on $\cS$ which satisfies the cocycle property
over the metric dynamical system $(\Omega,\cF,\mu,\theta)$ (see
Section 2.1). $\cA$ is called a {\em finite-state random dynamical
system} ({\em finite-state RDS} for short) or a {\em countable-state
random dynamical system} ({\em countable-state RDS} for short), when
$\cS$ is a finite set or a countable set  respectively. 

The cocycle $\cA$ admits a unique {\em matrix representation} $\cM,$ called the {\em induced linear cocycle}:
\begin{eqnarray}\label{M}
\cM(n,\omega)&=&(\cM_{s_is_j}(n,\omega)),\nonumber\\
\cM_{s_is_j}(n,\omega)&=&
\begin{cases}1,&s_i=\cA(n,\omega)s_j,
\\0,&\text{otherwise},
\end{cases}
\;\, i,j=1,2,\cdots,
\end{eqnarray}
where $n\in\NN,\ \omega\in\Omega$. We will show in Section~2  that $\cM$
is indeed a linear cocycle over $(\Omega,\cF,\mu,\theta)$ acting on
$\RR^k$, if $\#\cS=k$,  and on $\ell^1$, if $\#\cS=\infty$ (see
Lemmas~\ref{composition property},~\ref{cocycle}).

In parallel to continuous-state RDS theory, the framework of
dtds-RDS is more practical and has enjoyed a wide range of
applications in science and engineering \cite{MQY,YWQ}.
Particularly fitting examples include the random and probabilistic
Boolean networks. The random Boolean network, introduced in 1969 by
S. Kauffman \cite{Kau1,Kau2} as a simple model for gene regulatory
networks, has two state variables $0,1$ representing ``off'' and
``on'' states of a gene respectively. A network of genes evolves according to a given
Boolean function. A random Boolean network concerns randomly chosen
initial data \cite{Dro}, while a probabilistic Boolean network
involves  randomly chosen i.i.d. Boolean functions to build a
``randomly chosen constituent network'' with deterministic Boolean
dynamics in a ``random period of time'' \cite{SD}.   Since then,
different forms of Boolean network dynamics has found applications
in neural computations, gene networks, as well as Boltzmann machines
that led the current deep learning \cite{Ackley,GeQian08, Hop,LLOT,Zhang}.

To describe a more general random or stochastic Boolean
network, we let $\cS$ be the set of Boolean variables on the nodes
of a network, $\Omega$ be the probability space assembling all
possible randomness, noise, or stochasticity in the network, and
$\alpha(\omega): \mathcal S\to\mathcal S$, $\omega\in\Omega$ be a
measurable family of state transition maps determined by a set of Boolean functions describing the connectivity
of the network. Then
\[
{\cA}(n,\omega)=\alpha(\theta^{n-1}\omega)\circ \cdots \circ
\alpha(\theta(\omega))\circ \alpha(\omega), \qquad n\in \NN
\]
defines the random cocycle of the corresponding dtds-RDS and the 
and the induced linear cocycle $\cM$ are the adjacency matrices of the
network. With such a general setting, not only do we allow any
finite or countable number of state variables, but also the
randomness involved is also made general which particularly allows
the dependency on the past history. We remark that unlike the case
of a discrete stochastic process which emphasizes intrinsic
stochasticity in the movement of each and every individual, the
dtds-RDS modeling approach emphasizes the randomness in the ``law of
motion" for an entire population of individuals which are governed
by the same law \cite{MQY}.

When   $\cA$ is a finite-state RDS, it generates a random subshift
of finite type of the $\cS$-symbolic random skew-product flow over
$(\Omega,\cF,\mu,\theta)$, with induced cocycle $\cM$ being the
random transition matrices. Indeed, let $\Sigma$ denote the set
$\cS^{\NN}$ of all sequences of elements of $\cS$ endowed with the
product topology, together with the left-shift operator $T$. Then
\[
\Sigma_{\cA}(\omega)=\{(s_0,s_1,\cdots): s_i\in\cS,\;
\cM_{s_is_{i+1}}(1,\theta^i\omega)=1\},\qquad \omega\in\Omega
\]
become a random
 subshift  of finite type in $\Sigma\times \Omega$, i.e., $T^n\Sigma_\mathcal A(\omega)\subseteq\Sigma_\mathcal A(\theta^n\omega)$,
$n\in\NN$. We note that when $\cA$ is not a finite-state RDS, a
similar generating symbolic dynamical system is undefined and thus
the cocycle $\cA$ is a more general description of the corresponding
dtds-RDS. We note that the above notion of a countable-state RDS
allows the dynamical consideration of a wide range of applications
including random lattices and infinite networks.

A well-developed concept in the random dynamical systems theory is
synchronization (see \cite{Bax, FGS,  LeJan, Newman}
and reference therein), which is intimately related to the ``random attractor" in
random  as well as in non-autonomous dynamical systems. Roughly
speaking, synchronization describes the phenomenon that for almost
surely two different initial states collapse into a single one after
sufficiently long time. This related property also has interests in
neuron biology. In \cite{Lin}, synchronization is well discussed for
neuron networks, which are formulated by a system of stochastic
differential equations with white noise which is considered as the
stimulus to the network. If along almost each single stimulus
realization the response of the network always remains the same
independent of which initials it starts from (i.e., synchronized),
then the neuron network is said to be reliable. Under general
conditions, the sign of the maximal Lyapunov exponent is set as a
criterion for the reliability of the network: the negativity of
maximal Lyapunov exponent implies reliability and positivity implies
unreliability \cite{Lin}.


In \cite{YWQ}, the notion of synchronization for a dtds-RDS  is
introduced and investigated under the spirit that an i.i.d dtds-RDS
is a more refined model for stochastic systems than its  counter
part - the Markov chain. Adopting this notion to the present
setting, given a dtds-RDS $\cA$ over $(\Omega,\cF,\mu,\theta)$ and  $\omega\in\Omega,$ a pair $\{s, s^\prime\}\subset\mathcal S$ is said to
\emph{$\omega$-synchronize}  if there
exists $n(\omega)\in\mathbb N$ such that
$$\mathcal A(n,\omega)s=\mathcal A(n,\omega)s^\prime,\ \;\,
n\geq n(\omega).$$
$\cA$ is said
to {\em synchronize} if for $\mu$-a.e. $\omega\in\Omega,$ any pair $\{s,s'\}\subset\cS$ $\omega$-synchronizes.


In this paper, refining and generalizing results in \cite{YWQ}, we
give an equivalent characterization of  synchronization for a
dtds-RDS from the viewpoint of the multiplicative ergodic theory.
Our main results state as follows.
\medskip

\noindent{\bf Theorem A.} {\em Consider a finite-state RDS $\cA$
with the $k$-state set $\mathcal S=\{s_1,\cdots,s_k\}$. Then the
following holds:
\begin{itemize}
\item[(i)] For any $\omega\in\Omega,$ the Lyapunov exponents
 of the induced linear cocycle $\mathcal M$ acting on $\RR^k$ exist and take at least the
 value $\lambda_1(\omega)=0$  and possibly another value
 $\lambda_2(\omega)=-\infty.$
 \item[(ii)] For any $\omega\in\Omega,$ there exists a partition of $\mathcal S:$
     $$\eta(\omega)=\{W_1(\omega),\cdots, W_{m_1(\omega)}(\omega)\},$$
     where $m_1(\omega)$ is the multiplicity of the Lyapunov
     exponent $\lambda_1(\omega)=0$,
     such that a pair $\{s, s^\prime\}\subset \cS$ $\omega$-synchronizes if and only if  $s,s^\prime\in W_i(\omega)$
     for some integer $1\le i\le m_1(\omega).$

\item[(iii)] The dtds-RDS $\cA$ synchronizes if and only if for $\mu$-a.e. $\omega\in\Omega,$ $\cM$ admits precisely two Lyapunov exponents
$\lambda_1(\omega)=0$, $\lambda_2(\omega)=-\infty$ with respective multiplicities
$m_1(\omega)=1,\ m_2(\omega)=k-1$.

\end{itemize}
}

\medskip
For each $\omega\in\Omega$, the $\omega$-synchronized partition
$\eta(\omega)$ stated in Theorem A~(ii) can be of course defined
through the equivalence relation that $s_i\sim s_j$ in $\cS$ if and
only if $\{s_i,s_j\}$ $\omega$-synchronizes. We note that not only
does Theorem A~(ii) characterize the number of equivalence classes
but also the proof of it gives constructive descriptions of the
synchronized subsets $\{W_i(\omega)\}$ explaining the mechanism of
partial synchronization.

In applying Theorem A to a particular finite-state RDS $\cA,$ of importance is the characterization of the multiplicity  $m_1(\omega)$ of the $0$ Lyapunov exponent rather than the Lyapunov exponent itself because it is always attained by (i) above. One useful approach in making such a characterization for a particular model with ergodic $\mu$  is to combine Theorem A with the multiplicative ergodic theorem (see Theorem \ref{MET}) to show that $m_1(\omega)$ is a constant for $\mu$-a.e. $\omega$. This then gives a characterization on whether $\cA$ is synchronized or otherwise its  number of synchronized subsets. We will demonstrate such applications in Section 5 with a p53 random network.

For a countable-state RDS, the $\omega$-synchronized partition and
subsets can be defined through the same equivalence relation, but
similar constructive descriptions of synchronized subsets are not
available. In addition, unlike the finite-state case, Lyapunov
exponents for a countable-state RDS need not exist in general, and
even they do in some special situation, there can be uncountably
many of them (see Remark~\ref{countable}). Nevertheless, we have the
following result.

\medskip

\noindent{\bf Theorem B.} {\em Consider a countable-state RDS $\cA$
with the state set  $\mathcal S=\{s_i:\ i\in\NN\}$. Then the
following holds:
\begin{itemize}
\item[(i)] For any $\omega\in\Omega,$ let
     $$\eta(\omega)=\{W_i(\omega):i\in\mathcal I_\omega\}$$
     be the $\omega$-synchronized partition of $\cS$,
     where $\mathcal I_\omega$ is finite or countable and for each  $i\in\mathcal I_\omega$,
     $W_i(\omega)$ is a synchronized subset, i.e., $s, s^\prime\in
W_i(\omega)$ if and only if
 $\{s, s^\prime\}\subset \cS$  $\omega$-synchronizes.  Then
$$\overline{\{\boldsymbol{v}\in\ell^1:\lambda(\omega,\boldsymbol{v})=-\infty\}}=\{\boldsymbol{v}\in\ell^1:
\sum\limits_{j: s_j\in W_i(\omega)}v_j=0, \ \forall i\in\mathcal
I_\omega\},$$ where $\lambda(\omega,v)$ is the Lyapunov exponent of
the induced linear cocycle $\mathcal M$ associated with $\omega,v$
defined in \eqref{LE-for-countable}.

\item[(ii)] The dtds-RDS $\cA$ synchronizes if and only if for $\mu$-a.e. $\omega\in\Omega$ the Lyapunov exponent $-\infty$ is attained
and
\begin{eqnarray*}
\overline{\{\boldsymbol{v}\in\ell^1:\lambda(\omega,\boldsymbol{v})=-\infty\}}
=\left\{\boldsymbol{v}\in\ell^1:\sum\limits_{i=1}^{+\infty}v_i=0\right\}.
\end{eqnarray*}
\end{itemize}
}
\medskip

We note that Theorem B (ii) implies that if a countable-state RDS
$\mathcal A$ synchronizes,  then the closure of the spectral
subspace associated with the Lyapunov exponent $-\infty$ has
codimension-$1$ for $\mu$-a.e. $\omega\in\Omega.$

The paper is organized as follows. In Section \ref{RDS and MET}, we
study the induced linear cycles and their ergodic properties. In
particular,  Lyapunov exponents are characterized for the
finite-state case and Theorem A~(i) is proved. Basic notions of
random dynamical systems and a general multiplicative ergodic
theorem are also recalled. Section \ref{sec:synchronization for
finite RDS} is devoted to the analysis of synchronization phenomenon
for a finite-state dtds-RDS. We give a characterization of full
synchronization with respect to the Lyapunov exponent $0$ and a
characterization of partial synchronization with respect to
synchronized subsets. Theorem A~(ii),~(iii) are proved. Similar
analysis is conducted in Section \ref{sec:syn for countable state
space} for the countable-state case. In particular, we give a
characterization of full synchronization with respect to the
spectral subspace of the Lyapunov exponent $-\infty$ and a
characterization of partial synchronization with respect to
synchronized subsets. Theorem~B~(i),~(ii) are proved.
In Section 5, we give an example of probabilistic Boolean networks, i.e., the p53 random network, to demonstrate applications of Theorem A. Some discussions on similar applications in more complicated random networks are also given. 

\section{Ergodic properties of dtds-Random dynamical systems}\label{RDS and MET}
In this section, we study basic ergodic properties of the dtds-RDS,
and in particular, we show a multiplicative ergodic theorem for the 
induced linear cocycle. Notions of random dynamical systems,
cocycles,  and the classical multiplicative ergodic theorem will be
recalled for the case of discrete time variable.

\subsection{Matrix representations as a cocycle} Let $(\Omega,\mathcal F, \mu)$
be a probability space, where $\mu$ is a  probability measure
defined on the $\sigma$-algebra $\cF$ of $\Omega$.  $(\Omega,\mathcal F, \mu, \theta)$
 is called a {\em metric dynamical system} if $\theta:\Omega\to\Omega$ is a
measurable, measure-preserving transformation, i.e.,
$\mu(\theta^{-n}B)=\mu(B)$, $n\in \NN$, $B\in \cF$. Let $X$ be a
topological space, called {\em state space}. A {\em cocycle over
$(\Omega,\mathcal F, \mu, \theta)$} acting on $X$ is a family $\{\Phi(n,\omega),\ n\in
\NN,\ \omega\in\Omega\}$ of continuous mappings on $X$ which are measurable
in $\omega$ and satisfy the following {\em  cocycle property}:
\begin{eqnarray*}
\Phi(0,\omega)=id,\ \ \Phi(m+n,\omega)=\Phi(n,\theta^m \omega)\circ\Phi(m,\omega), \;\,
m,n\in\mathbb N,\; \omega\in\Omega.
\end{eqnarray*}
With the cocycle property, it is clear that the mapping
$\phi:\NN\times X\times Y\to X\times Y$:
$$
\phi(n,x,\omega)=(\Phi(n,\omega)x,\theta^n \omega),\qquad n\in\NN,\, x\in X,\, \omega\in\Omega
$$
defines a measurable skew-product flow on the phase space $X\times
\Omega$, called a {\em random dynamical system}. If the state space $X$
is a normed vector space and $\Phi(n,\omega)$ is a bounded linear
operator for each $n\in\NN, \omega\in\Omega$, then $\Phi$ is called a {\em
linear cocyle} over $(\Omega,\mathcal F, \mu, \theta)$ acting on $X$.

\vspace{2mm}

Now consider the dtds-RDS  $\cA$ described   in Section~1. Let $\cM$
be the matrix representation of $\cA$ defined in \eqref{M}.
\medskip

\begin{Lem}\label{composition property} $\cM$ depends on $\omega\in\Omega$ measurably and  satisfies the cocycle
property over $(\Omega,\cF,\mu,\theta)$.
\end{Lem}
\begin{proof} Denote $\cS=\{s_i\}$ and let $\Gamma=\{f:\mathcal
S\rightarrow\mathcal S\}$ be the collection of all deterministic self-maps of $\cS.$  For  each $f\in \Gamma,$ define  $M_f=:( (M_f)_{ij})$ as follows
\begin{equation}\label{Mf}
 (M_f)_{ij}=
\begin{cases}1,&s_i=fs_j,
\\0,&\text{otherwise},
\end{cases}
\ \; i,j=1,2,\cdots.
\end{equation}
It is easy to see that $\{M_f\}$ satisfies
\begin{equation}\label{matrix cocycle}
M_{f_1}\cdot M_{f_2}=M_{f_1\circ f_2},\quad f_1,f_2\in\Gamma.
\end{equation}
Note that for fixed $n\in\NN$ and $\omega\in\Omega,$ the matrix representation $\cM$ in  \eqref{M} satisfies $\cM(n,\omega)=M_{\cA(n,\omega)}.$ The lemma now follows from \eqref{matrix cocycle}, the measurability
of $\cA$ in $\omega$, and   the cocycle property of $\cA$.
\end{proof}

The following lemma shows that $\cM$ is a linear cocycle acting on a suitable state space.

\begin{Lem}\label{cocycle}
Let $\cA$ be a dtds-RDS with matrix representation $\cM.$ 

(i) If $\cA$ is a finite-state RDS with $k$-state set, then 
$\cM$ is a linear cocycle over $(\Omega,\cF,\mu,\theta)$ acting on  $\RR^k;$

(ii) If $\cA$ is a countable-state RDS, then $\cM$ is a linear cocycle over $(\Omega,\cF,\mu,\theta)$ acting on $\ell^1:=\{\boldsymbol {v}=(v_i)_{i=1}^{+\infty}:\sum\limits_{i=1}^{+\infty}|v_i|<+\infty\}.$
\end{Lem}
\begin{proof} 
(i)	Since for each $n\in\NN,$ $\cM(n,\omega)$  is a $k\times k$ matrix, $\cM$ is a linear cocycle acting on  $\RR^k$ by noting that any $k\times k$ matrix  is a bounded linear operator on $\RR^k$ with respect to the standard Euclidean norm.
	
(ii) We only need
to verify that for any fixed $n\in\NN$ and $\omega\in \Omega,$ $M_{\cA(n,\omega)}$  defined in \eqref{M} is a bounded linear operator on $\ell^1.$ For any $\boldsymbol {v}=(v_i)_{i=1}^{+\infty}\in \ell^1$, denote
$(u_i)_{i=1}^{+\infty}=\boldsymbol u=M_{\cA(n,\omega)}\boldsymbol{v}$. Then
\begin{eqnarray*}
u_i=\begin{cases}
0, &\cA(n,\omega)s_j\neq s_i\ \text{for\ all}\ j\geq1\\
\sum\limits_{j:\cA(n,\omega)s_j=s_i}v_j, &\text{otherwise},
\end{cases}
\end{eqnarray*}
and therefore
\begin{eqnarray}\label{norm inequality}
\sum\limits_{i=1}^{+\infty}|u_i|=\|\boldsymbol
u\|_1=\sum\limits_{i=1}^{\infty}|\sum\limits_{j:\cA(n,\omega)s_j=s_i}v_j|
\leq\sum\limits_{i=1}^{+\infty}|v_i|<+\infty.
\end{eqnarray}
\end{proof}
 

\begin{Rmk}  We note that for a countable-state RDS $\cA,$ the induced linear cocycle $\cM$ needs not define a cocycle acting on the
state space $\ell^p$ for $1<p\leq+\infty$.  For instance, if for some $n\in\NN$ and $\ \omega\in\Omega,$ 
it satisfies that $\cA(n,\omega)s_i=s_1$ for all $i=1,2,\cdots$.  Then  $\boldsymbol
v=(1,1/2,\cdots,1/n,\cdots)\in\ell^p$ for any $1<p\leq+\infty$, but
\[M_{\cA(n,\omega)}\boldsymbol v=(1+1/2+\cdots+1/n+\cdots,0,\cdots)\notin\ell^p\]
because $\sum\limits_{n=1}^{+\infty}\dfrac{1}{n}=+\infty.$
 
\end{Rmk}

\subsection{Multiplicative ergodic properties of dtds-RDS}
Following the celebrated  work of Oseledec's \cite{Oseledets},
multiplicative ergodic theory has been substantially developed for
linear cocycles. Below, we state a version of multiplicative ergodic
theorem on finite dimensional state space. Denote
$\log^+(\cdot)$=$\max\{\log(\cdot), 0\}.$

\begin{Thm}\label{MET}{\rm(Theorem 3.4.1 in \cite{Arnold})} Consider a linear cocycle
$\Phi$ over a metric dynamical system $(\Omega,\mathcal F, \mu, \theta)$ acting on
the state space $\RR^k.$ Assume that $\log^+\|\Phi(1,\cdot)\|\in L^1(\Omega,\mathcal F, \mu).$ Then there exist an integer-valued,
measurable function $r,$ real-valued, measurable functions
$\{\lambda_i\}_{i=1}^r$ with $\lambda_r$ possibly being $-\infty,$
integer-valued, measurable functions $\{m_i\}_{i=1}^r$ with
$\sum\limits_{i=1}^rm_i=k,$ and a measurable filtration $\mathbb
R^k=V_1\supsetneqq\cdots\supsetneqq V_{r}\supsetneqq
V_{r+1}=\emptyset$ such that for $\mu$-a.e. $\omega\in\Omega,$ the following
holds.
\begin{itemize}
\item[{\rm (i)}] {\rm (}invariance{\rm )} $r(\theta\omega)=r(\omega)$,
$\lambda_i(\theta\omega)=\lambda_i(\omega)$, $m_i(\theta\omega)=m_i(\omega)$, \ $i=1,\cdots,r(\omega)$;

\item[{\rm (ii)}] {\rm (}dimensionality{\rm )} $\dim V_{i}(\omega)-\dim
V_{i+1}(\omega)=m_{i}(\omega),\;\,i=1,\cdots,r(\omega);$

\item[{\rm (iii)}] {\rm (}exponential growthness{\rm )} For any
$\boldsymbol{v}\in V_{i}(\omega)\backslash V_{i+1}(\omega),\;
i=1,\cdots,r(\omega),$
\begin{equation*}\label{Lya}
\lim\limits_{n\rightarrow+\infty}\dfrac{1}{n}\log\|\Phi(n,\omega)\boldsymbol
v\|=\lambda_{i}(\omega).
\end{equation*}
\end{itemize}
Moreover, if $\mu$ is ergodic, then all $\lambda_i(\omega)$'s, $m_i(\omega)$'s and $r(\omega)$ are constants for $\mu$-a.e. $\omega\in\Omega.$
\end{Thm}
\medskip

Quantities $\lambda_i, m_i$, $i=1,\cdots, r$ are referred to as the
{\em Lyapunov exponents} and their {\em multiplicities},
respectively. For cocycles acting on a Banach space, Lyapunov
exponents and their multiplicities can be similarly defined,
provided that they exist.
\medskip

When restricting to the case of finite-state RDS, we have the
following refined result for the Lyapunov exponents which implies
Theorem~A~(i). 

\begin{Prop}\label{MET for finite dsdt-RDS} Let $\cA$ be a finite-state RDS and $\cM$ be the induced linear cocycle over
$(\Omega,\cF,\mu,\theta)$ acting on the state space $\RR^k$ equipped
with the standard Euclidean norm $\|\cdot\|$, where $k=\#\cS$. Then
for any $\omega\in\Omega$ and any $\boldsymbol{v}\in\mathbb R^k,$
\begin{eqnarray*}
\lambda(\omega,\boldsymbol{v}):=\lim\limits_{n\rightarrow+\infty}\dfrac{1}{n}\log\|\mathcal
M(n,\omega)\boldsymbol{v}\|
\end{eqnarray*}
exists and equals either $0$ or $-\infty$, and consequently, the
Lyapunov exponents of $\cM$  take at most two values $\lambda_1=0$
or $\lambda_2=-\infty.$ Moreover, the Lyapunov exponent
$\lambda_1=0$ is always attained.
\end{Prop}
\begin{proof} We note that with $n,\omega$ varying, there are only a
finite number of choices of $\cM(n,\omega)$. It follows that for any
$\boldsymbol v\in\RR^k,$ $\|\mathcal M(n,\omega)${\boldmath$v$}$\|$
only take a finite number of different values for all $n\in\mathbb
N, \ \omega\in\Omega.$ Now we let $\omega\in\Omega$ be fixed.

If $\mathcal M(n,\omega)${\boldmath$v$}$\neq\mathbf 0$ for all
$n\in\mathbb N$, then $\|\mathcal M(n,\omega)${\boldmath$v$}$\|$
have uniform positive upper and lower bounds, i.e., there exist
constants $0<\kappa_1<\kappa_2<+\infty$ independent of $n,\omega$
such that $\kappa_1\leq\|\mathcal
M(n,\omega)${\boldmath$v$}$\|\leq\kappa_2$ for all $n\in\mathbb N.$
Thus, $\lim\limits_{n\rightarrow+\infty}\dfrac{1}{n}\log\|\mathcal
M(n,\omega)${\boldmath$v$}$\|$ exists and equals $0.$

If there exists $n\in\mathbb N$ such that $\mathcal
M(n,\omega)\boldsymbol{v}=\boldsymbol0$, then $\mathcal
M(m,\omega)\boldsymbol v=\mathbf 0$ for any $m\geq n.$ Thus
$\lim\limits_{n\rightarrow+\infty}\dfrac{1}{n}\log\|\mathcal
M(n,\omega)${\boldmath$v$}$\|=-\infty.$

We now argue that the Lyapunov exponent $\lambda_1=0$ is
always attained. For otherwise, there exists an $n\in\mathbb N$ such
that $\mathcal M(n,\omega)\boldsymbol v=0$ for all $\boldsymbol
v\in\mathbb R^k$. This is impossible since $\mathcal M(n,\omega)$ is
not a zero matrix.

\end{proof}

\medskip

In the case of countable-state RDS, for any $\omega\in\Omega$ and any $\boldsymbol v\in\ell^1,$
define
\begin{eqnarray}\label{LE-for-countable}
\lambda(\omega,\boldsymbol v)=\limsup\limits_{n\rightarrow+\infty}\dfrac{1}{n}\log\|\cM(n,\omega)\boldsymbol{v}\|_{1},
\end{eqnarray}
which is referred to as the {\em Lyapunov exponent} of $\cM$
associated with $\omega, \boldsymbol v,$ whenever the limit exists.

\begin{Rmk}\label{countable} 

(i) Unlike  Theorem \ref{MET}, the conclusion of Proposition \ref{MET for finite dsdt-RDS} holds for all $\omega\in\Omega$ instead of a full $\mu$-measure set.

(ii)  There have been works in infinite-dimensional multiplicative ergodic theorem \cite{Blu, FGQ, GQ, LianLu}. However, all these works are for cases of countably many Lyapunov exponents.
In our case, the induced linear cocycle of a countable-state RDS can  admit  uncountably many  Lyapunov exponents. As an example, let
$$\mathcal A(n,\omega)=f^n, \;\; n\in\NN,\, \omega\in\Omega
$$
where $f: \cS=\{s_i, i\in\NN\}\to \cS$ is a deterministic map such
that $fs_i=s_{i-1}$ for all $i\geq2$ and $fs_1=s_1.$ For any  given
$\lambda\in(0,1),$ choose
\[
 \boldsymbol v=\left(\dfrac{\lambda}{1-\lambda}, -\lambda, -\lambda^2, \cdots, -\lambda^n,\cdots\right)\in\ell^1.
\]
Then for the induced linear cocycle $\cM,$
\begin{eqnarray*}
\lambda(\omega, \boldsymbol v)&=&
\limsup\limits_{n\rightarrow+\infty}\dfrac{1}{n}\log\|\cM(n,\omega)\boldsymbol{v}\|_{1}\\
&=&\lim\limits_{n\rightarrow+\infty}\dfrac{1}{n}\log\left(\dfrac{\lambda}{1-\lambda}-\dfrac{\lambda(1-\lambda^n)}{1-\lambda}+
\dfrac{\lambda^{n+1}}{1-\lambda}\right)
\\
&=&\lim\limits_{n\to+\infty}\dfrac{1}{n}\log\left(\dfrac{2\lambda^{n+1}}{1-\lambda}\right)
\\
&=&\log\lambda.\\
\end{eqnarray*}
Thus,  for this example,  any value in $(-\infty, 0)$ is a Lyapunov
exponent of $\cM$.

\end{Rmk}

\section{Synchronization in finite-state RDS}\label{sec:synchronization for finite RDS}

In this section, we study synchronization phenomenon for a
finite-state RDS $\cA$ with state set $\cS=\{s_i:\ i=1,2,\cdots,
k\}$. Recall that the induced linear cocycle $\cM$ over
$(\Omega,\cF,\mu,\theta)$ acts on the state space $\RR^k$ equipped
with the standard Euclidean norm.

\subsection{A necessary and sufficient condition for
synchronization}  By Proposition~\ref{MET for finite dsdt-RDS}, for
any $\omega\in\Omega,$ $\cM$ admits at most two Lyapunov exponents:
$\lambda_1(\omega)=0, \lambda_2(\omega)=-\infty$.
Let \begin{equation}\label{V}
V(\omega)=\{\boldsymbol{v}\in\mathbb R^k:
\lambda(\omega,\boldsymbol{v})=-\infty\}. 
\end{equation}
It is easy to see that $V(\omega)$ is a linear subspace  of $\mathbb
R^k$.  The following result says  that $V(\omega)$ is actually
contained in a co-dimension-$1$ hyperplane $E_0.$

\medskip

\begin{Lem}\label{E0} For any $\omega\in\Omega,$
\begin{equation}\label{E000}
V(\omega)\subseteq E_0=:\{{\boldsymbol v}={(v_1,\cdots,v_k)^\top}
:\sum\limits_{i=1}^kv_i=0\}.
\end{equation}
\end{Lem}
\begin{proof} Let ${\boldmath v}={(v_1,\cdots,v_k)^\top} \in V(\omega)$.
Then  there exists $n\in\mathbb N$ such that $\mathcal
M(n,\omega)\boldsymbol{v}=\boldsymbol 0$. If we denote $\mathcal
M(n,\omega)${\boldmath$v$}={\boldmath$u$}$=(u_1,...,u_k)^\top$, then
$\sum\limits_{i=1}^ku_i=\sum\limits_{i=1}^kv_i$. Hence ${\boldmath
v}\in E_0$.
\end{proof}

Note that for any $\boldsymbol v\in \mathbb R^k\backslash
V(\omega),$ we have $\lambda(\omega,\boldsymbol v)=0.$ Denote
$m_2(\omega)=\dim V(\omega)$ and
\begin{eqnarray}\label{m_1}
m_1(\omega)=\dim\mathbb R^k-\dim V(\omega)=k-m_2(\omega).
\end{eqnarray}
As in Theorem \ref{MET}, we call $m_1(\omega), m_2(\omega)$
multiplicities of $\lambda_1(\omega),\lambda_2(\omega)$
respectively. From Proposition~\ref{MET for finite dsdt-RDS}, we
know that
\begin{eqnarray}\label{m_1 and m_2}
m_1(\omega)\geq1,\ \  m_2(\omega)\leq k-1.
\end{eqnarray}

\medskip

The following result, from which Theorem A~(iii) follows, shows that
the synchronization of $\cA$ happens exactly when the equality in
\eqref{m_1 and m_2} holds true.


\begin{Thm}\label{characterization for synchronization}
 $\mathcal A$ synchronizes if and only if for $\mu$-a.e. $\omega\in\Omega,$ $\cM$ admits precisely two Lyapunov exponents
$\lambda_1(\omega)=0$, $\lambda_2(\omega)=-\infty$ with respective multiplicities
\begin{equation}\label{m1}
m_1(\omega)=1,\ \  m_2(\omega)=k-1.
\end{equation}
\end{Thm}
\begin{proof}
Suppose $\mathcal A$ synchronizes. Then for $\mu$-a.e.
$\omega\in\Omega,$ there are integers  $\ n(\omega)\in\mathbb N$ such
that for every $ m\geq n_(\omega),$ we can find an integer $1\le
\ell(\omega,m)\le k$ such that
$$ \mathcal A(m,\omega)s_i=s_{\ell(\omega,m)}, \,\; i=1,\cdots,k.$$
For given $\omega$ and $ m\geq n(\omega),$ it follows from \eqref{M}
that the matrix $\mathcal M(m,\omega)$ has every entries on the
$\ell(\omega,m)$-th row being $1$ and  all other being $0$. Let
$\boldsymbol v\in E_0,$ where $E_0$ is the co-dimension-$1$
hyperplane defined in \eqref{E000}. We then have $\mathcal
M(m,\omega) {\boldmath v}=\textbf0$, i.e., ${\boldmath v}\in
V(\omega)$, where $V(\omega)$ is the spectral subspace defined in
\eqref{V}. Hence $E_0\subseteq V(\omega)$, and by Lemma~\ref{E0}, we
actually have
\begin{eqnarray}\label{-infty=E_0}
V(\omega)=E_0.
\end{eqnarray}
It follows that $m_2(\omega)={\rm dim} \ V(\omega)={\rm dim} \
E_0=k-1$ and $\lambda_2=-\infty$ is attained.  By
Proposition~\ref{MET for finite dsdt-RDS},  the Lyapunov exponent
$\lambda_1=0$ is always attained with $m_1(\omega)=k-m_2(\omega)=1.$

Now suppose \eqref{m1} holds. Then ${\rm dim} \ V(\omega)=k-1$,
$\mu$-a.e. $\omega\in\Omega$. It follows from  Lemma~\ref{E0} that
\eqref{-infty=E_0} holds for $\mu$-a.e. $\omega\in\Omega.$  Let $s_i,\ s_j$ be any two distinct
elements of $\cS$ and denote by $e_i$, respectively $ e_j$, the
$i$-th, respectively the $j$-th, standard unit vector in $\RR^k$.
Since $e_i-e_j\in E_0, $ we have by \eqref{-infty=E_0} that
$e_i-e_j\in V(\omega)$ for $\mu$-a.e. $\omega\in\Omega$. It follows
from \eqref{V} that, for a fixed such $\omega$, there exists
$n(\omega)$ sufficiently large such that for all $n\ge n(\omega)$,
$$\mathcal M(n,\omega)(e_i-e_j)=\textbf0,  $$ i.e., $$\mathcal M(n,\omega)e_i=\mathcal M(n,\omega)e_j,$$ or equivalently, $$\mathcal
A(n,\omega)s_{i}=\mathcal A(n,\omega)s_{j}.$$ This show that any
pair of elements in $\cS$ synchronizes, hence $\mathcal A$
synchronizes.
\end{proof}


\subsection{Synchronization along synchronized subsets}\label{sec:syn-subset} It can be seen from the
definition of  synchronization that if the cocycle $\cA$ is
invertible, then none pair of elements in $\cS$ can  synchronize.
Thus, for a dtds-RDS, total non-synchronization  and total
synchronization are two extreme situations, and  partial
synchronization is to be expected in general. Below, we give a
characterization of the mechanism of partial synchronization for a
finite-state RDS $\cA$.

We call $\xi=\{W_1,\cdots, W_p\}$  a \emph{partition} of $\mathcal S$ if
each $W_i$, referred to as a {\em component of $\xi$}, is a
subset of $\mathcal S$, $W_i\cap W_j=\emptyset$ for
all $ i\neq j$, and $\bigcup\limits_{i=1}^p W_i=\mathcal S$.
 Let $\mathcal P_\mathcal S$ be the set of all partitions of
$\mathcal S.$ For $\xi,\eta\in \mathcal P_\mathcal S$, we say
\begin{eqnarray*}
\xi\preccurlyeq\eta\ \Longleftrightarrow\ \text{each component of
$\xi$ is the union of some components of $\eta$}.
\end{eqnarray*}
It is clear that the binary relation ``$\preccurlyeq$" defines a
partial ordering on $\mathcal P_\mathcal S$. A partition family
$\mathcal Q\subseteq\mathcal P_\mathcal S$ is called a \emph{chain}
if for any $\xi,\eta\in\mathcal Q,$ either $\xi\preccurlyeq\eta $ or
$\eta\preccurlyeq\xi.$ We call $\eta\in\mathcal Q$ a \emph{minimal
partition} of $\mathcal Q$ if $\xi\preccurlyeq\eta$ and
$\xi\in\mathcal Q$ imply that $\xi=\eta.$ It is a well-known fact
that any finite chain admits a unique minimal partition.

For the sake of analyzing partial synchronization occurred in the
dtds-RDS, we would like to consider partitions of $\cS$ that are
connected to the random cocycle $\cA$. For any $\omega\in\Omega,$ it is easy to see that
$\xi_n(\omega)=:\{\cA^{-1}(n,\omega)s_i:i=1,...,k\}$ is a partition
of $\cS$ for each $n\in\NN$.
\medskip

\begin{Lem}\label{xi_n} For each $\omega\in\Omega,$ $\{\xi_n(\omega):\ n\in\NN\}$ form
a nonincreasing chain of $\mathcal P_\mathcal S$, i.e.,
$$\xi_0(\omega)\succcurlyeq\xi_1(\omega)\succcurlyeq\cdots\succcurlyeq\xi_n(\omega)\succcurlyeq\cdots,$$
which is in fact a finite chain.
\end{Lem}
\begin{proof} For fixed  $n\in\NN$, let $W$ be a component of $\xi_{n+1}(\omega)$ and denote
$$
\cS_n=\cA^{-1} (1,\theta^n\omega)(\cA(n+1,\omega)W).
$$  It follows
from the the cocycle property of $\cA$ that
\[
W=\cA^{-1}(n,\omega)\left( \cA^{-1}
(1,\theta^n\omega)(\cA(n+1,\omega)W)\right)=\bigcup_{s\in
\cS_n}\cA^{-1}(n,\omega)s.
\]
Since $W$ is arbitrary and each $\cA^{-1}(n,\omega)s$,
$s\in\cS_n$, is a component of $\xi_n(\omega)$, we have
$\xi_{n}(\omega)\succcurlyeq\xi_{n+1}(\omega)$.

Since the chain $\{\xi_n(\omega):\ n\in\NN\}$ is nonincreasing and
$\cS$ is finite, their number of components is a constant for $n$
sufficiently large. Hence $\{\xi_n(\omega):\ n\in\NN\}$ is a finite
chain.
\end{proof}

Given $\omega\in\Omega,$ let $\eta(\omega)$ be the minimal partition of the chain $\{\xi_n(\omega):\
n\in\NN\},$ which we refer to as the {\em $\omega$-synchronized partition}. Components of $\eta(\omega)$ are called
\emph{$\omega$-synchronized subsets} of $\mathcal A.$

\begin{Prop}\label{syn space} Consider the finite-state RDS $\cA$.
Then for any $\omega\in\Omega,$ a pair $\{s_i,s_j\}$ in $\cS$
	$\omega$-synchronizes if and only if  $s_i,s_j$ lie in a same
	$\omega$-synchronized subset of $\cA.$ 
\end{Prop}
\begin{proof} By definition of synchronization, if $\{s_i,s_j\}$ $\omega$-synchronizes, then
	$\mathcal A(n,\omega)s_i=\mathcal A(n,\omega)s_j$ for $n$
	sufficiently large. It follows that $s_i$ and $s_j$ lie in a same
	component of $\xi_n(\omega)$ for $n$ sufficiently large.
	Consequently,  $s_i$ and $s_j$ lie in a same $\omega$-synchronized
	subset of $\cA.$ The converse is also clear.
\end{proof}

\begin{Rmk}
(i) Actually, by {\em Proposition \ref{syn space}}, we can define the partition $\eta(\omega)$ in a
more straightforward way through the equivalence relation that $s_i,
s_j\in\mathcal S$ lie in the same component of $\eta(\omega)$ if and
only if $\{s_i,s_j\}$ $\omega$-synchronizes. We derive it through a
chain of decreasing partitions as in {\rm Lemma~\ref{xi_n}} to give
a more intuitive sense of how partial synchronization happens in a
finite-state RDS.

(ii) We note that the concept of  $\omega$-synchronization for a dtds-RDS $\cA$ is defined in a pairwise way. 
In fact, for the finite-state case, this concept is global, i.e., for a given $\omega\in\Omega,$ if $W(\omega)$ is a $\omega$-synchronized subset, then there exists $n(\omega)\in\NN$  such that $\cA(n,\omega)W(\omega)$ is a single state for all $n\ge n(\omega).$ 


\end{Rmk}

The following proposition relates the cardinality of the $\omega$-synchronized partition $\eta(\omega)$ and the multiplicity of Lyapunov exponent $\lambda_1(\omega)=0.$

\begin{Prop}\label{M1}For each $\omega\in\Omega,$
\begin{eqnarray}\label{number of eta}
\#\eta(\omega)=m_1(\omega),
\end{eqnarray}
where $m_1(\omega)$ is the multiplicity of the Lyapunov exponent
$\lambda_1(\omega)=0.$
\end{Prop}
\begin{proof} For each $\omega\in\Omega$, since for all $n$
sufficiently large, the number of components of  $\xi_n(\omega)$
 is a constant, it must equal to $\#\eta(\omega).$ For a such
$n$ sufficiently large,  we note that $\#\xi_n(\omega)$ equals to
$(k-m_2(\omega)),$ where $m_2(\omega)$ is the multiplicity of the
Lyapunov exponent $\lambda_2(\omega)=-\infty.$ In fact,  by
definition of Lyapunov exponents, a vector $\boldsymbol{v}\in\mathbb
R^k$ such that $\lambda_2(\omega,\boldsymbol{v})=-\infty$ if and
only if $\boldsymbol{v}$ satisfies a homogeneous system of linear
equations with coefficient matrix being $\cM(n,\omega).$ By \eqref{M},  the $i$-th row of $\cM(n,\omega)$ is non-zero if and only if there exists some $j$ such that $\cA(n,\omega)s_j=s_i,$ i.e., $\cA(n,\omega)^{-1}s_i$ is non-empty, which, by the construction of $\xi_n(\omega),$ is a component of $\xi_n(\omega).$
Thus, the number of non-zero rows equals $\#\xi_n(\omega).$  Note that for
each column of $\cM(n,\omega),$ there is only one entry being
non-zero, which means that  the number of non-zero rows is exactly
the rank of $\cM(n,\omega).$ Let $V(\omega)$ be as in ~\eqref{V}.
Then its dimension equals $(k-\#\xi_n(\omega)),$ i.e.,
$m_2(\omega)=k-\#\eta(\omega).$ \eqref{number of eta} now follows
from \eqref{m_1}.
\end{proof}

\begin{Rmk}\label{Rmk:finite}
We note that  {\em Theorem \ref{characterization for
synchronization}} is a special case of {\em
Propositions~\ref{syn space},~\ref{M1}} when  $\#\eta(\omega)=1$ for 
$\mu$-a.e. $\omega\in\Omega.$
\end{Rmk}

Theorem~A (ii) follows from Proposition \ref{syn space}
and \ref{M1}.



\section{synchronization in countable-state RDS}\label{sec:syn for countable state space}

In this section, we study the synchronization phenomenon for a
countable-state RDS $\cA$ with state set $\cS=\{s_i:\ i\in\NN\}$.
Recall that the induced linear cocycle $\cM$ over
$(\Omega,\cF,\mu,\theta)$ acts on the state space $\ell^1$ equipped
with the $\ell^1$-norm which we denote by $\|\cdot\|_1$.

\subsection{A necessary and sufficient condition for
synchronization}
 Let
\begin{equation}\label{E00}
        E_0=
        \left\{\boldsymbol{v}\in\ell^1:\sum\limits_{i=1}^{+\infty}v_i=0\right\},
\end{equation}
and \[
F_0=\Big\{\boldsymbol v\in E_0: \boldsymbol v\ \text{
has at most finitely many components being non-zero}\Big\}.
\]
We have the following basic facts.

\medskip

\begin{Lem}\label{basic facts} The following holds.
\begin{itemize}

\item[{\rm (1)}] $E_0$ is a closed subspace of $\ell^1.$

\item[{\rm(2)}] $\overline{F_0}=E_0.$
\end{itemize}
\end{Lem}

\begin{proof}
(1) It is obvious that $E_0$ is a subspace of $\ell^1.$ To show the
closeness, we take any sequence $\{\boldsymbol{v_n}\}$ in $E_0$ such
that $\boldsymbol{v_n}=(v_i^{(n)})_{i=1}^{+\infty}\rightarrow\boldsymbol v=(v_i)_{i=1}^{+\infty}$.
Since
$$0=\lim\limits_{n\to+\infty}\sum\limits_{i=1}^{+\infty}v_i^{(n)}
=\sum\limits_{i=1}^{+\infty}\lim\limits_{n\to+\infty}v_i^{(n)}=\sum\limits_{i=1}^{+\infty}
v_i,$$ we have  $\boldsymbol v\in E_0.$

(2) For any $\boldsymbol{v}=(v_i)_{i=1}^{+\infty}\in E_0,$ let
$\boldsymbol u^{(n)}=(u^{(n)}_{i})_{i=1}^{+\infty}$ be such that
\begin{equation*}u^{(n)}_i=
\begin{cases}v_i,&1\leq i<n,
\\\sum\limits_{j=n}^{+\infty}v_j,&i=n,
\\0,&i>n.
\end{cases}
\end{equation*}
Then $\boldsymbol u^{(n)}\in F_0$ and it is easy to check that
$\|\boldsymbol u^{(n)}-\boldsymbol v\|_{1} \to0$ as
$n\rightarrow+\infty.$
\end{proof}

We directly obtain Theorem~B (ii) from the following result by the definition of synchronization for a countable-state RDS $\mathcal A$.
\medskip

\begin{Thm}\label{thm:countable-syn} Consider the countable-state
RDS $\cA$  and $\omega\in\Omega.$ Then any pair $\{s_i,s_j\}\subset\cS$ $\omega$-synchronizes if and only if
\begin{eqnarray}\label{sE0}
\overline{\{\boldsymbol{v}\in\ell^1:\lambda(\omega,\boldsymbol{v})=-\infty\}}=E_0,
\end{eqnarray}
where $\lambda(\omega,v)$ is  defined in \eqref{LE-for-countable}.
\end{Thm}
\begin{proof} Suppose  \eqref{sE0} holds but
there exist $ s_{i_0}, s_{j_0}\in\cS$ with $s_{i_0}\neq
s_{j_0}$ satisfying
\begin{equation}\label{An}\mathcal A(n,\omega)s_{i_0}\neq\mathcal
A(n,\omega)s_{j_0},\;\; \, \forall n\geq0.
\end{equation}

Since $e_{i_0}-e_{j_0}\in E_0$
and \eqref{sE0} holds, there exists
$\boldsymbol{v}=(v_i)_{i=1}^{+\infty}\in\ell^1$ such that
\begin{eqnarray}
& &\|\boldsymbol v-(e_{i_0}-e_{j_0})\|_{1}<\dfrac{1}{8},\label{distance}\\
& &\lambda(\omega, \boldsymbol
v)=\lim\limits_{n\to+\infty}\dfrac{1}{n}\log\|\mathcal
M(n,\omega)\boldsymbol{v}\|_{1}=-\infty. \label{LE}
\end{eqnarray}
We note by (\ref{distance}) that
$$v_{i_0}\in(\dfrac{7}{8},\dfrac{9}{8}),\ v_{j_0}\in(-\dfrac{9}{8},-\dfrac{7}{8}),\ \sum\limits_{i\neq i_0,j_0}|v_i|<\dfrac{1}{8}.$$
It then follows from \eqref{An} that  $$\|\mathcal
M(n,\omega)\boldsymbol
v\|_{1}=\sum\limits_{i=1}^{+\infty}|\sum\limits_{j:\mathcal
A(n,\omega)s_j=s_i}v_j|\geq|v_{i_0}|+|v_{j_0}|-\sum\limits_{i\neq
i_0,j_0}|v_i|>\dfrac{13}{8},$$ which leads to a contradiction  to
(\ref{LE}).

Conversely, suppose any pair $\{s_i,s_j\}\subset\cS$ $\omega$-synchronizes. We first show that
\begin{eqnarray}\label{contained 1}
\overline{\{\boldsymbol{v}\in\ell^1:\lambda(\omega,\boldsymbol{v})=-\infty\}}\subseteq
E_0.
\end{eqnarray}

If not, then there exists $\boldsymbol
v=(v_i)_{i=1}^{+ \infty}\in\ell^1\backslash E_0$ such that
$\lambda(\omega,\boldsymbol v)=-\infty.$ Since
$|\sum\limits_{i=1}^{+\infty}v_i|=:a>0$,  and any pair $\{s_i,s_j\}$ synchronizes, there
exists $m>0$ such that
\begin{eqnarray*}
\|\mathcal M(n,\omega)\boldsymbol
v\|_{1}\geq|\sum\limits_{i=1}^mv_i|-\sum\limits_{i=m+1}^{\infty}|v_{i+1}|>\dfrac{1}{2}a,
\end{eqnarray*}
for all $n$ sufficiently large, which is a contradiction to the fact
that $\lambda(\omega,\boldsymbol v)=-\infty.$ Thus
$$\{\boldsymbol v\in\ell^{1}: \lambda(\omega,\boldsymbol v)=-\infty\}\subseteq E_0.$$
(\ref{contained 1}) then follows from  Lemma \ref{basic facts} (1).

Next, we show that
\begin{eqnarray}\label{contained 2}
E_0\subseteq\overline{\{\boldsymbol{v}\in\ell^1:\lambda(\omega,\boldsymbol{v})=-\infty\}}.
\end{eqnarray}


Since any pair $(s_i,s_j)$ $\omega$-synchronizes, for any $\boldsymbol v\in
F_0,$ we have
$$\|\mathcal M(n,\omega)\boldsymbol v\|_{1}=|\sum\limits_{i: v_i\neq0}v_i|=0,$$
when $n$ sufficiently large. Thus $F_0\subseteq\{\boldsymbol{v}\in\ell^1:\lambda(\omega,\boldsymbol{v})=-\infty\}$ and
hence
$$\overline{F_0}\subseteq\overline{\{\boldsymbol{v}\in\ell^1:\lambda(\omega,\boldsymbol{v})=-\infty\}}.$$
Now \eqref{contained 2} follows
from Lemma~ \ref{basic facts} (2).

\end{proof}

\begin{Rmk}
For a finite-state RDS $\cA$,  it follows from the proof of {\rm
Theorem \ref{characterization for synchronization}} that if $\cA$
synchronizes, then $$E_0=\{\boldsymbol v\in\mathbb R^k:
\lambda(\omega,\boldsymbol v)=-\infty\},\;\, \mu-a.e.\
\omega\in\Omega,
$$  where $E_0$ is the co-dimension-$1$
hyperplane defined in \eqref{E000}. However, for a countable-state
RDS that synchronize, a similar identity is no longer true, i.e., it
is necessary to take closure in the identity \eqref{sE0}. To see
this, consider the example  in {\rm Remark~\ref{countable}}. It is
easy to see that the cocycle $\cA$ in this example synchronizes but
for any $\omega\in\Omega$,
\[
\{\boldsymbol v\in\ell^1: \lambda(\omega,\boldsymbol v)=-\infty\}
\]
 is not of co-dimension-$1$ because there are more than two other
 Lyapunov exponents. Hence it cannot equal to the hyperplane $E_0$ defined in \eqref{E00}.
\end{Rmk}
\subsection{Synchronization along synchronized  subsets}
In the case of countable-state RDS $\cA,$ for each
$\omega\in\Omega,$ the $\omega$-synchronized partition
$\eta(\omega):=\{W_i(\omega): i\in\mathcal I_\omega\}$ of $\cS$ can
be defined through the equivalence relation as mentioned in
Remak~\ref{Rmk:finite} (ii), i.e., a pair $(s_i, s_j)$ belongs to a
same component of $\eta(\omega)$ if and only if $(s_i,s_j)$
$\omega$-synchronizes. We still call each component of $\eta(\omega)$ as a {\em $\omega$-synchronized subset} of $\cA.$ 
Differing from the finite-state case, the number of synchronized subsets
$\mathcal I_\omega$ can be infinite. Also, when restricted to each $\omega$-synchronized subset $W(\omega),$ $\cA(n,\omega)W(\omega)$ need not be a single state for any finite $n.$

Now by restricting Theorem
\ref{thm:countable-syn} to a same $W_i(\omega)$ of $\eta(\omega),$
we have the following result which gives Theorem B~(i).

\medskip

\begin{Thm}\label{thm:countable-partial}
Let $\mathcal A$ be a countable-state RDS. Then for each $\omega\in\Omega,$
$$\overline{\{\boldsymbol{v}\in\ell^1:\lambda(\omega,\boldsymbol{v})=-\infty\}}=\{\boldsymbol{v}\in\ell^1:
\sum\limits_{j: s_j\in W_i(\omega)}v_j=0, \ \forall i\in\mathcal
I_\omega\},$$ where $\{W_i(\omega): i\in\mathcal I_\omega\}$ are the
synchronized subsets of $\mathcal A$ and $\lambda(\omega,v)$ is
defined in \eqref{LE-for-countable}.
\end{Thm}

\begin{Rmk}
Theorem~\ref{thm:countable-syn} is a special case of Theorem~\ref{thm:countable-partial} when $\#\mathcal I_\omega=1.$
\end{Rmk}



\section{Applications to i.i.d random networks}
In this section, we demonstrate some applications of our theoretical findings  to certain biological, i.i.d random networks in describing their synchronization behaviors, in particular in determining
their number of synchronized subsets. In fact, as the example below will show,  our results can be applied to networks with more general external randomness, e.g., those modeled by Poisson noises if the intensity is very small. Let $\cS$ denote a discrete state set and  $\Gamma$ collect all maps on $\cS,$ together with a probability measure $Q$ on $\Gamma.$ We recall that an i.i.d dtds-RDS is that each time a map  from $\Gamma$ is randomly chosen to act on $\cS$  according to  $Q.$ The metric dynamical system $(\Omega,\cF,\mu,\theta)$ modeling the noise is simply defined by the probability space $(\Omega,\cF,\mu)=\prod_{0}^\infty(\Gamma,2^\Gamma,Q)$ together with the left-shift operator $\theta.$ 

We shall consider a particular i.i.d random network - the probabilistic Boolean network model for the biochemical dynamics of regulating
protein p53 in biological cells, followed by some discussions in treating i.i.d networks with more complexity.



\subsection{The random p53 network }\label{sec:p53}

It has been shown that the tumor suppressor, p53,  is a crucial
protein in multicellular organism that prevents cancer
develoment \cite{PK2012}. The working mechanism of p53, proposed by
Harris and Levine \cite{LH2005}, is described by a negative
feedback loop as shown in Figure \ref{feedback-loop}. In response to
an external stress signal, the cell cycle enters arrest, apoptosis,
cellular senescence, and DNA repair \cite{LH2005, PK2012}.
These events were modeled in \cite{GeQian2009} by a Boolean network. Without
the external stress, p53 is in the low steady state and the network
is determined by another set of Boolean functions \cite{GYG2017}. 
We assume an entire population of cells simultaneously experience a 
same external stress signal that is fluctuating.  The dynamics
of p53 in different cells then can be modeled by a dtds-RDS.

\begin{figure}\centering
\includegraphics[width=0.35\linewidth]{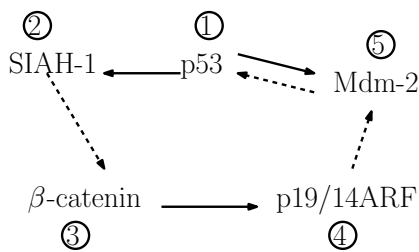}
\caption{The negative feedback loop of p53. 
Each gene is denoted as node 1-5.  The solid arrows denote stimulatory interactions, whereas the dashed arrows denote inhibitory influences.}
\label{feedback-loop}
\end{figure}

To describe the p53 dynamics using dtds-RDS, we consider the
external stress of the cell as the extrinsic noise, which comes
from, for instance, the DNA damage due to environmental 
radiation.  It can be properly modeled as a discrete-time
Poisson process with small intensity $\lambda$ as follows. Initially, all
cells start from different initial conditions and without external
stress, the p53 dynamics of all cells follow a map $B.$ Once the
external stress appears after time random $T$ with exponential 
distribution $\text{Exp}(\lambda)$, p53
dynamics of all cells follow a different map $A$ (Figure \ref{state}{\bf (A)}) and cells
will engage in the DNA repair process for a constant period of time
$T_r,$  where  $T_r$ is much smaller than the expected waiting time $1/\lambda$. Denote  the map $C$ as the $T_r$-th iteration of the map $A.$ Afterwards all cells return back to normal and follow the map
$B$ (Figure \ref{state}{\bf (B)}) again until another external stress appears.
It is possible that another external stress appears during the
repair cycle, but the probability of this happening is usually very
small, and even this happens, the cell may enter the cycle of
apoptosis. The discrete-time counterpart of Poisson process for small $\lambda$ is a Bernoulli process, e.g., a Bernoulli shift in the language of dynamical systems. More
precisely,  the noise probability space is simply
\[(\Omega,\cF,\mu)=\prod_{0}^\infty(\{C,B\},2^{\{C,B\}},\{p,1-p\}),\]
where  $Q:=\{p,1-p\}$ denotes the probability measure on $\{C,B\}$ with
$C$ and $B$ taking the measure $p$ and $1-p,$ respectively. Note that $p\approx \lambda$ if we discretize the time of the Poisson process by the unit time.
Let $\theta$ be the left-shift map on $\Omega.$ Then the product measure
$\mu=\{p,1-p\}^{\NN_0}$ is an ergodic $\theta$-invariant probability
measure on $(\Omega,\cF)$ (\cite[Theorem 1.12]{Wal}), and consequently $(\Omega, \cF,\sigma,\mu)$ is a metric  dynamical system.

\begin{figure}
	\begin{tabular}{cc}
		{\bf (A)}  \includegraphics[valign=t, width=0.32\linewidth]{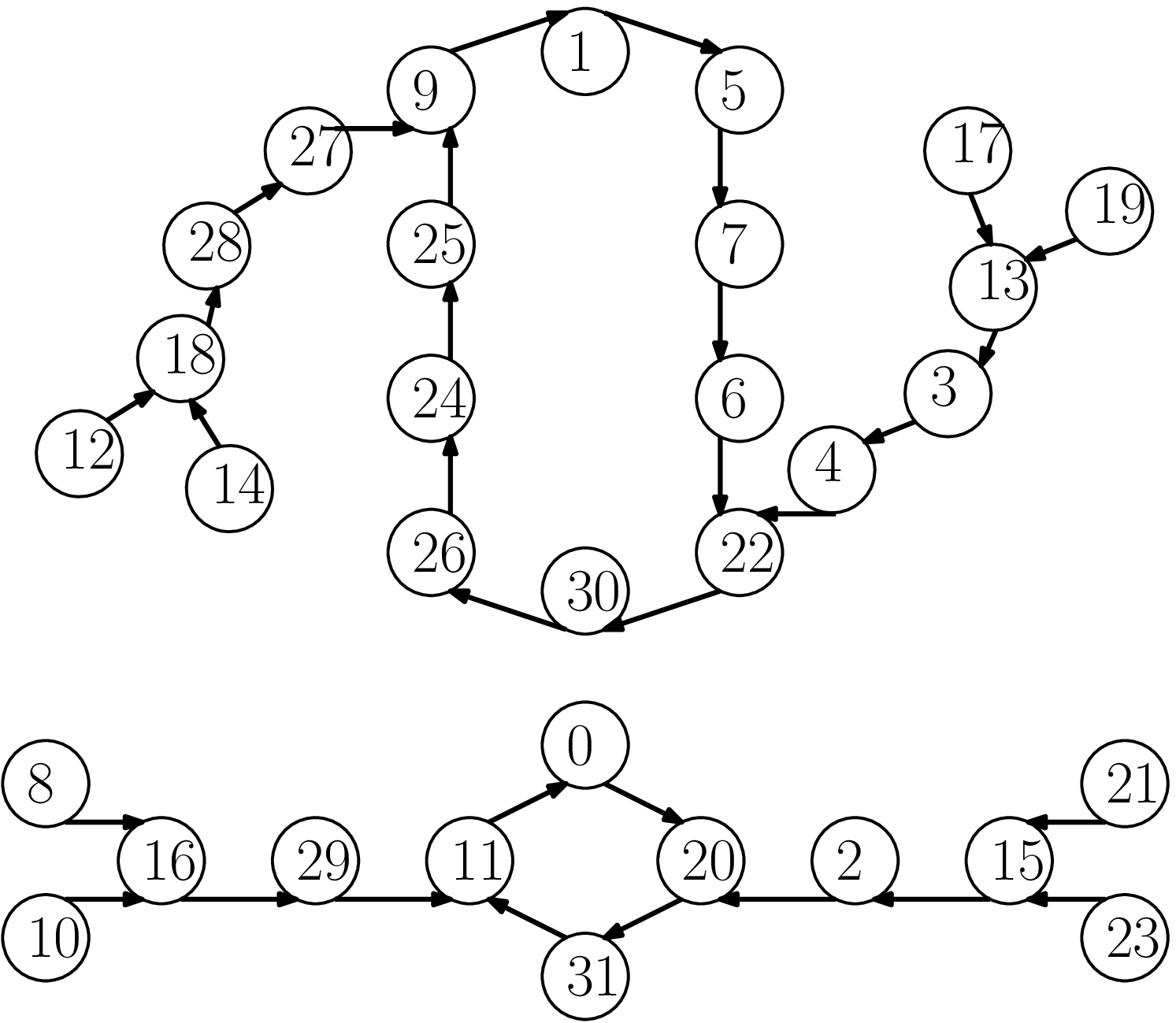} &{\bf (B)}   \includegraphics[valign=t, width=0.32\linewidth]{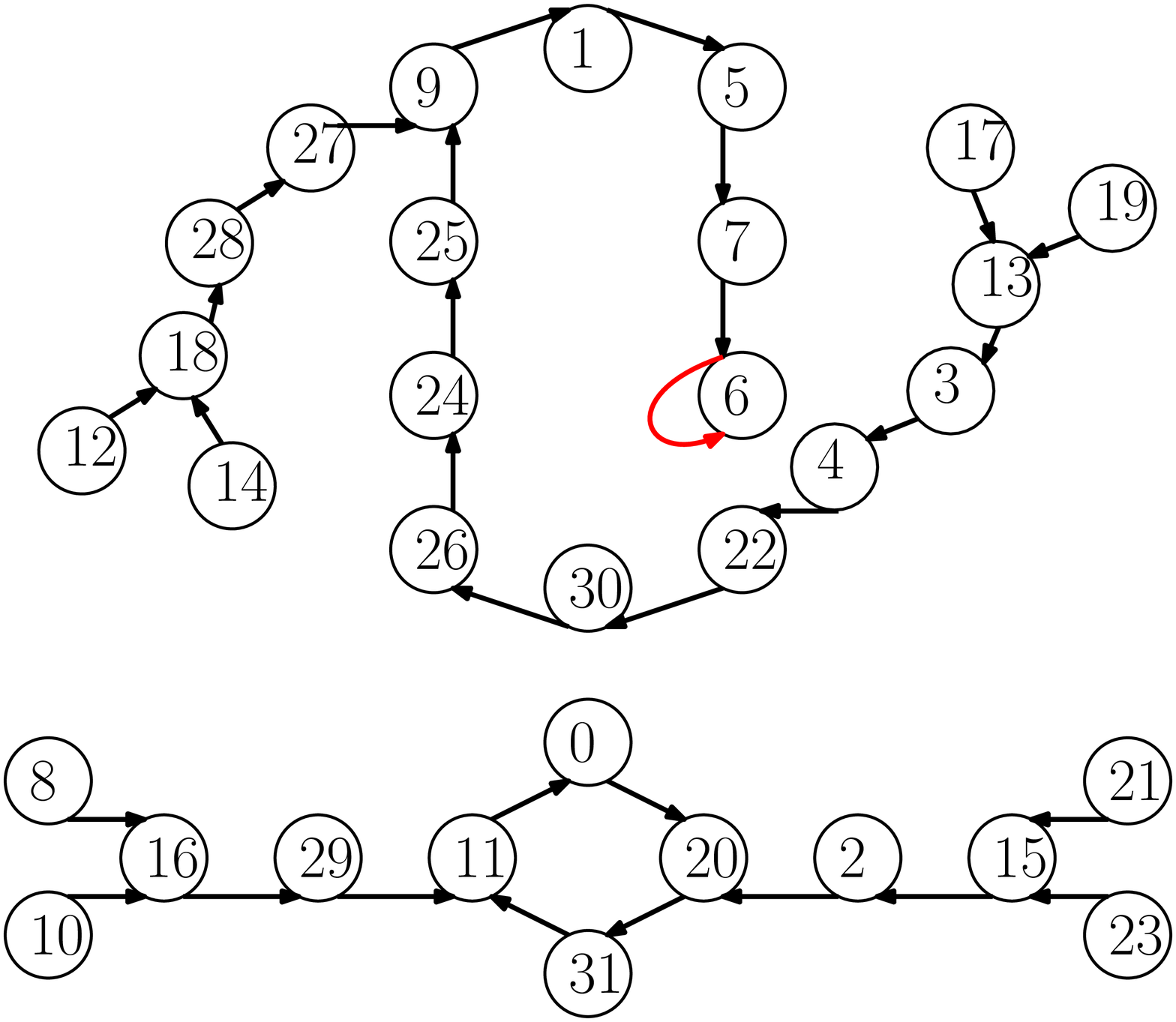}
	\end{tabular}
	\caption{ (A) The state transition map for the Boolean network corresponding to p53 dynamics in response to an external stress. (B) The state transition map corresponding to p53 dynamics in the absence of the external stress. }
	\label{state}
\end{figure}

In the p53 network as shown in Figure \ref{feedback-loop}, there are five nodes
in total and each node only admits two values, $1$ or $0,$ representing the active and inactive state respectively.  So the state set $\mathcal{S}$ is the binary expansion from 0 to 31 and in total, there are 32 states. In the state transition maps $A$ and
$B$, we use decimal numbers 0-31 to indicate the gene expression. Note that we do not plot the map $C$ which is just $A^{T_r}.$

For each $\omega\in\Omega,$ define $\cA: \cS\to\cS$ as follows
\[
\cA(n,\omega)=\cA(1,\theta^{n-1}\omega)\circ\cdots\circ\cA(1,\omega),\qquad
n\in\NN,
\]
where
\begin{eqnarray*}
\cA(1,\omega)=\begin{cases}
C,&{\text{if the $0-$th position of $\omega$ is $C$};}\\
B,&{\text{if the $0-$th position of $\omega$ is $B$}.}
\end{cases}
\end{eqnarray*}
 It is easy to see that $\cA$ satisfies the cocycle property and
hence it is an i.i.d  dtds-RDS over $(\Omega, \cF,\mu,\theta)$.

Figure \ref{state}{\bf (A)} shows that the map  $A$ admits two attractors
which are in fact two limit cycles:
\begin{eqnarray}
& &1\rightarrow 5 \rightarrow 7 \rightarrow 6 \rightarrow 22 \rightarrow 30 \rightarrow 26 \rightarrow 24 \rightarrow 25 \rightarrow 9 \rightarrow 1,\label{limit_cycle1}\\
& &0 \rightarrow 20 \rightarrow 31 \rightarrow 11 \rightarrow 0,\label{limit_cycle2}
\end{eqnarray}
respectively.
Note that the map $C$ has the same attractors as the map $A.$
However, under  the map $B,$ the  limit cycle \eqref{limit_cycle1}
collapses into a fixed point $\{6\},$ which indicates the
homeostasis of the cell, whereas the limit cycle
\eqref{limit_cycle2} remains the same.




The  dtds-RDS $\cA$ does not synchronize, because two nodes from two
different basins of attractions, like 7 and 0, will never
synchronize. However, according to Theorem A(ii), $\cA$ is always
partially synchronized, i.e., for any $\omega\in\Omega,$ there
exists a partition $\eta(\omega)$ of $\mathcal S$ such that a pair
of states belonging to the same component of $\eta$ is
$\omega$-synchronized, and moreover, the cardinality of
$\eta(\omega)$ equals the multiplicity $m_1(\omega)$ of the Lyapunov
exponent $\lambda_1(\omega)=0.$


Using Theorem A(ii), we have the following result.
\medskip

\begin{Prop}\label{prop example} For $\mu$-a.e.
$\omega\in\Omega$, $m_1(\omega)=5$, i.e., $\eta(\omega)$ consists of
$5$ $\omega$-synchronized subsets.
\end{Prop}

\begin{proof} Let
\[
\Omega_*=\{\omega\in\Omega: \text{the $0$-th position of
$\theta^n\omega$ is $B$ for $n=0,\cdots, 14$}\}.
\]
By the definition of $\mu,$ we have
\[
\mu(\Omega_*)=(1-p)^{15}>0.
\]

It is easy to see that for any $\omega\in\Omega,$ there
 are at least five components of $\eta(\omega),$ and four of them constitutes the attracting basin of the
 limits cycle \eqref{limit_cycle2}, which are $W_1=\{31,29,21,23\}, W_2=\{0,2,8,10\}, W_3=\{ 11,15 \}$ and
 $W_4=\{ 20,16\},$ respectively. If $\omega\in\Omega_*$, then
 the map $B$ is applied repeatedly  for 14 times to all the 20 states in the basin $W_5$ of the limit cycle
 \eqref{limit_cycle1}, driving all these states into the state $6$.
 Consequently, any two points in $W_5$ is $\omega$-synchronized. It
 follows from  Theorem A(ii) that
 \[
m_1(\omega)=5,\qquad \omega\in\Omega_*.
 \]
 Since $\mu(\Omega_*)>0$ and $\mu$ is ergodic, it follows from Theorem~\ref{MET} that  $m_1(\omega)=5$ for $\mu$-a.e. $\omega\in\Omega.$

\end{proof}

We note that, for some other $\omega$,  
if, as $n$ increases, switchings between maps $A$ and $B$ at the
$0$-th position of $\theta^n\omega$ is rather frequent, then it may
happen that certain two different states in the limit cycle
\eqref{limit_cycle1}  never collapse, i.e.,  $W_5$ can be further
decomposed  into different synchronized subsets and $m_1(\omega)>5.$
But Proposition~\ref{prop example} says that such $\omega$'s are of zero
$\mu$-measure.

 Proposition~\ref{prop example}
 allows one to define,  up to a $\mu$-null set of
$\omega$'s, the equivalence relation of these gene expressions that
$s_i\sim s_j$, if and only if they are in the same component of
$\eta(\omega).$ To
 our knowledge, such concepts are new to the community of Boolean
networks. The extrinsic noise-induced synchronization behaviour described above is in principle different from synchronization in mechanical systems (e.g. coupled-oscillators), because no direct interactions between cells in our model is required.
It would be interesting to conduct a biological experiment to verify the partial synchronization phenomenon of multiple cells with different gene expressions exposed by the same radiation source.

 	In many realistic models of gene regulations, random perturbations are  incorporated by making the network never ``get stuck'' in any state set.  Generally speaking, a random gene perturbation means that any given gene has a small probability of being randomly flipped to a different state, e.g.,  a gene admitting value 0 can be flipped to value 1, although this only happens with very small probability \cite{SD}.
Under random gene perturbations, a partially synchronized random network usually becomes a synchronized one since states from different attracting basins may collapse together due to the random flipping effect. In fact, this can be understood from Theorem A  that the  number of synchronized subsets is equal to the multiplicity of $0$ Lyapunov exponent,  which, if being bigger than 1, is  easily  reduced to 1 under certain generic perturbations.





\subsection{More general i.i.d networks}

In the example of p53 random network, the number of state set is only 32.  However, a general i.i.d random network in reality can have tens of thousands state variables.
In general, to determine the multiplicity of $0$ Lyapunov exponent and the synchronized partition for a complex random network with a huge number of state variables along certain infinite sequence of maps are very  difficult tasks. Nevertheless, we show below that the cardinality of the synchronized partitions, i.e. the multiplicity of $\lambda_1=0$ by Theorem A(ii), for a general i.i.d dtds-RDS can be estimated, at least numerically, by using the Markov chain it induces. 
We recall from \cite{YWQ} that an i.i.d dtds-RDS  uniquely induces a Markov chain on $\cS$ with transition probability $P=(p_{ij})_{1\le i,j\le k},$ where $p_{ij}=Q\{f\in\Gamma:fs_i=s_j\}$ and  $\cS=\{s_1,\cdots,s_k\}.$

An upper bound of the multiplicity of $\lambda_1$ can be estimated by the number of recurrent states of the induced Markov chain as follows.
Given a transit state, say $s_i,$ of the Markov chain, almost all sample trajectories of the Markov chain visit $s_i$ only finite times. Then by our construction of the synchronized subsets of $\{\eta(\omega)\}$ in section \ref{sec:syn-subset}, along  almost every sequence of maps determined by the element $\omega$ in $\Omega,$ the component of $\eta(\omega)$ corresponding to the pre-image of $s_i,$ $\cA^{-1}(n,\omega)s_i,$ is an empty set  for $n$ sufficiently large. Consequently,  the number of non-trivial components of each $\eta(\omega)$ is no more than the number of recurrent states of the induced Markov chain.

As to the lower bound, we will show that the multiplicity of $\lambda_1=0$ is no less than the number of recurrent classes of the induced Markov chain. In fact,  algorithms has been developed to figure out the latter ones, 
e.g.,	if we treat this  induced Markov chain as the digraph $G=(V,E)$, then an efficient algorithm of depth-first search of the digraph could be used to identify the recurrent communicating classes in $\mathcal{O}(|V|+|E|)$ time \cite{Jarvis}.
Recall that a recurrent class of a Markov chain is the set of all recurrent states that can go to each other with positive possibilities. In other words,  any two initial states belonging to different recurrent classes will not collapse together along almost every sample trajectories of the Markov chain, i.e,  they belong to different synchronized subsets. Then Theorem A(ii) implies that the multiplicity of $\lambda_1$ is no less than the number of recurrent classes of the induced Markov chain. Inside each recurrent class, however, it is a more delicate issue to determine whether two different states belong to a same synchronized subset, for which, the two-point motion techniques (e.g. \cite{YWQ}) can be used. 
More precisely, by applying the same sequence of maps determined by an element $\omega$ in $\Omega,$  we  construct two infinite trajectories on $\cS$ starting from two different initial states. 
For the i.i.d dtds-RDS, this two-point motion induces a Markov chain on $\mathcal{S}\times \mathcal{S}$ with transition probability $W$ being \[W_{(s_i, s_j)\rightarrow (s_m,s_\ell)}=Q\{f\in\Gamma: fs_i=s_m, fs_j=s_\ell\}, \ \forall i,j,m,\ell\in\{1,\cdots,k\}.\]
Note that  when  $s_i=s_j, s_m=s_\ell$, the  transition probability is the same as that of the  Markov chain induced by the i.i.d dtds-RDS, and when $s_i=s_j, s_m\ne  s_\ell$, the transition probability is 0.
Therefore, $\{(s_1,s_1),  (s_2,s_2), \dots, (s_k,s_k)\}$ is a recurrent class the  Markov chain induced by $W.$ Furthermore, one can show that the i.i.d dtds-RDS  synchronizes if  and only if $\{(s_1,s_1),  (s_2,s_2), \dots, (s_k,s_k)\}$ is the only  recurrent class of the  Markov chain induced by the two-point motion. If there exists any other recurrent class, then for 
almost every $\omega$ in $\Omega,$  
the first state of any pair inside the class cannot be in the same component of the synchronized partition $\eta(\omega)$ as the second state.
 By Theorem A(ii),  the multiplicity of $\lambda_1=0$ within such a recurrent class for almost every $\omega$ is at least 2.  Indeed, we may better estimates the lower bound from this restriction. For instance, if $(s_i,s_j)$, $(s_j,s_\ell)$ and $(s_i, s_\ell)$ are in the same recurrent class other than the trivial one $\{(s_1,s_1),  (s_2,s_2), \dots, (s_k,s_k)\},$ then for almost every $\omega,$ $s_i, s_j$ and $s_\ell$ should be in three different components of $\eta(\omega),$ i.e., the multiplicity of $\lambda_1(\omega)=0$ in this case is at least 3.  In this way, we give a lower bound for the multiplicity of the $0$ Lyapunov exponent.  

 We illustrate the estimation by the following example. This example of 4 states comes from \cite{YWQ}. The  
deterministic maps to choose in the i.i.d dtds-RDS are 
\begin{align*}
  \Gamma=\left\{
    \underbrace{\left(\begin{array}{c}
        1\rightarrow 2   \\   2 \rightarrow 1 \\ 3\rightarrow 4 \\ 4 \rightarrow 3
          \end{array}\right)}_{\alpha_1},\  \underbrace{\left(\begin{array}{cc}
          1\rightarrow 1 \\    2\rightarrow 2 \\ 3\rightarrow 3 \\ 4 \rightarrow 4
          \end{array}\right)}_{\alpha_2}, \  \underbrace{\left(\begin{array}{cc}
       1\rightarrow 4\\   2\rightarrow 3 \\ 3\rightarrow 2 \\ 4 \rightarrow 1
          \end{array}\right)}_{\alpha_3},\  \underbrace{\left(\begin{array}{cc}
 1\rightarrow 3\\    2\rightarrow 4 \\ 3\rightarrow 1 \\ 4 \rightarrow 2 
          \end{array}\right)}_{\alpha_4} ,\  \underbrace{\left(\begin{array}{cc}
 1\rightarrow 1\\    2\rightarrow 2 \\ 3\rightarrow 3 \\ 4 \rightarrow 3 
          \end{array}\right)}_{\alpha_5} 
\right\}.
\end{align*}

Note that the maps $\alpha_1, \alpha_2, \alpha_3, \alpha_4$ are permutations, while  $\alpha_5$ is not. One can assign non-zero probability mass on each map such that the induced Markov chain is always aperiodic and irreducible. If we consider the two-point motion, then there exists a recurrent communicating class, $\{(1,3),(3,1),(2,4),(4,2),(2,3),(3,2),(4,1),(1,4)\},$ other than the trivial one. That is, if we start from any pair of these states, for some $\omega$, these two infinite long sequences will never synchronize under this i.i.d dtds-RDS. From our previous arguments, the first state cannot be in the same component as the second one in the partition $\eta(\omega),$ e.g.,  1 cannot be in the same component as 3. With this restriction, a possible minimal partition could be $\eta(\omega)=\left\{ \{1,2\},\{3,4\} \right\},$ which indicates that the multiplicity of $\lambda_1=0$ for some $\omega$ is at least 2. So we  give an estimation of the lower bound of the multiplicity.


\begin{thebibliography}{10}

\bibitem{Ackley}
D. H. Ackley, G. E. Hinton, and T. J. Sejnowski,
A learning algorithm for Boltzmann machine,
{\em Cog. Sci.}, {\bf 9}, 147--169, 1985.

\bibitem{Arnold}L. Arnold,
{\em Random Dynamical Systems}, Springer, New York, 1998.

\bibitem{Bax}P. H. Baxendale, Statistical equilibrium and two-point motion for a stochastic flow of diffeomorphisms,
{\em Spatial Stochastic Processes: Festschrift for T. E. Harris} (Ed. K. S. Alexander and J. C.
Watkins), Birkh\"aser, Boston,  189--218, 1991.


\bibitem{Blu} A. Blumenthal, A volume-based approach to the multiplicative ergodic theorem on Banach spaces, {\em Discrete Contin. Dyn. Syst.} {\bf 36}(5), 2377--2403, 2016.

\bibitem{Dro} B. Drossel,  Random Boolean Networks,
{\em Reviews of Nonlinear Dynamics and Complexity} (Ed. H. G,
Schuster), 69--110, 2008.

\bibitem{FGS} F. Franco, B. Gess, and M. Scheutzow, Synchronization by noise, {\em Prob. Th. \& Related Fields}, {\bf 168}, 511--556, 2017.


\bibitem{FGQ} G. Froyland, C. Gonz\'{a}lez-Tokman, and A. Quas, Hilbert space Lyapunov exponent stability,  {\em Trans. Amer. Math. Soc.}, 2019, to appear.

\bibitem{GeQian08} H. Ge, H. Qian, and M. Qian, Synchronized dynamics and nonequilibrium steady states in
a stochastic yeast cell-cycle network, {\em Math. Biosci.}, {\bf 211}, 132--152, 2008.

 \bibitem{GeQian2009}
\newblock H.~Ge and M.~Qian,
\newblock Boolean network approach to negative feedback loops of the p53 pathways: synchronized dynamics and stochastic limit cycles,
\newblock \emph{J. Comp. Biol.}, \textbf{16}(1), 119--132, 2009.

\bibitem{GQ}C. Gonz\'{a}lez-Tokman, A. Quas, 
A concise proof of the multiplicative ergodic theorem on Banach spaces, {\em J. Mod. Dyn.} {\bf 9}, 237--255, 2015.

\bibitem{GYG2017} \newblock Y. ~Guo, Z. ~You and H. ~Ge, \newblock  Robustness and relative stability of multiple attractors in  a stochastic Boolean network (in Chinese),
\newblock \emph{Sci. Sin. Math.}, \textbf{47}, 1831--1852,  2017. 

\bibitem{Hop}J. J. Hopfield, Neural networks and physical systems with emergent collective computational
abilities, {\em Proc. Natl. Acad. Sci.}, {\bf 79}, 2554--2558, 1982.

 \bibitem{LH2005}
\newblock S. L.~Harris and A. J. ~Levine,
\newblock The p53 pathway: positive and negative feedback loops,
\newblock \emph{Oncogene}, \textbf{24}, 2899--2908, 2005

\bibitem{LeJan}Y. L. Jan, \'{E}quilibre statistique pour les produits de diff\'{e}omorphismes
al\'{e}atoires ind\'{e}pendants, {\em Ann. Inst. Henri Poincare, Sect. B},
{\bf 23}, 111--120, 1987.

\bibitem{Jarvis} J. Jarvis and D. Shier,  Graph-theoretic analysis of finite Markov chains, in: D. Shier, K. Wallenius (Eds.), {\em Applied Mathematical Modeling: A Multidisciplinary Approach}, Chapman \& Hall/CRC, 2000. 

\bibitem{Kau1} S. A. Kauffman, Metabolic stability and epigenesis in randomly constructed genetic nets, {\em J. Theor. Biol.}, {\bf22}(3), 437--467, 1969.

\bibitem{Kau2} S. A. Kauffman, Homeostasis and differentiation in random genetic control networks, {\it Nature}, {\bf 224}, 177--178, 1969.

\bibitem{Kifer}Y. Kifer, {\em Ergodic Theory of Random Transformations.}, Birkh\"{a}user, Basel, 1986.


\bibitem{LianLu}Z. Lian and K. Lu, Lyapunov exponents and invariant manifolds for random dynamical systems in a Banach space, {\em Mem. Amer. Math. Soc. 206}, {\bf 967}, 2010.


\bibitem{Lin}K. K. Lin, Stimulus-response reliability of biological networks,  {\em Nonautonomous and Random Dynamical
Systems in Life Sciences}, Lecture Notes in Math. Biosci. (Ed. P.
Kloeden and C. P\"otzsche), 135--161, Springer, New York, 2012.

\bibitem{LLOT}F. Li, T. Long, Y. Lu, Q. Ouyang, and C. Tang, The yeast cell-cycle network is
robustly designed, {\em Proc. Natl. Acad. Sci.}, {\bf 101}, 4781--4786, 2004.

\bibitem{LQ}P. Liu and M. Qian, {\em Smooth Ergodic Theory of Random Dynamical Systems}, Springer, 2006.

\bibitem{MQY} Y. Ma, H. Qian, and F. X.-F. Ye, Stochastic dynamics: Models for intrinsic and extrinsic
noises and their applications (in Chinese), {\em Sci. Sin. Math.},
{\bf 47}, 1693--1702, 2017.


\bibitem{Newman} J. Newman, Necessary and sufficient conditions for stable synchronization inrandom dynamical systems, {\em Ergod. Theory Dyn. Syst.}, {\bf 38}(5), 1857--1875, 2018.

 \bibitem{Oseledets}V. I. Oseledec, A multiplicative ergodic theorem, {\em Trans. Mosc. Math. Soc.},  {\bf 19}(2), 179--210, 1968.
 
 \bibitem{PK2012}
 \newblock J. E. Purvis, K. W. Karhohs, C. Mock, E.  Batchelor,  A. Loewer and  G. Lahav,
 \newblock p53 dynamics control cell fate,
 \newblock \emph{Science}, \textbf{336}(6087): 1440--1444, 2012. 

\bibitem{SD}I. Shmulevich and E. R. Dougherty, {\em Probabilistic Boolean Networks: A Model for Gene Regulatory Networks},
SIAM Press, Philadelphia, PA, 2009.

\bibitem{Von} J. von Neumann, The general and logical theory of automata, {\em Cerebral mechanisms in behavior}, {\bf 1}(41), 1--2, 1951.

\bibitem{Wal}P. Walters,
{\em An Introduction to Ergodic Theory.}, Springer Verlag, 1982.

\bibitem{Wol} S. Wolfram, {\em A New Kind of Science}, Wolfram Media, IL, 2002.

\bibitem{YWQ}F. X.-F. Ye, Y. Wang, and H. Qian,
Stochastic dynamics: Markov chains and random transformations, {\em
Discrete $\&$ Cont. Dyn. Sys. B}, {\bf 21}, 2337--2361, 2016.

\bibitem{Zhang}
Y. Zhang, M.-P. Qian, Q. Ouyang, M. Deng, F. Li, and C. Tang,
Stochastic model of yeast cell-cycle network,
{\em Physica D}, {\bf 219}, 35--39, 2006.


\end{thebibliography}
\end{document}